\documentclass[12pt,a4paper,reqno]{amsart}
\usepackage{graphicx} 

\usepackage{amssymb}
\usepackage{amsmath}
\usepackage{amsthm}

\usepackage[all,knot]{xy}
\xyoption{arc} 
\usepackage{hyperref}
\usepackage{xcolor}
\newtheorem{definition}{Definition}
\newtheorem{proposicao}{Proposition}
\newtheorem{observacao}{Remark}
\newtheorem{exemplo}{Example}
\newtheorem{teorema}{Theorem}
\newtheorem{corolario}{Corollary}

\newcommand{\h}{\mathcal{H}}
\newcommand{\A}{\mathbf{A}}
\newcommand{\I}{\mathbf{I}}
\newcommand{\J}{\mathbf{J}}

\title{On the spectra of $k$-uniform threshold hypergraphs}

\author[M. Abd\'on]{Miriam Abd\'on}
\email{miriam\_abdon@id.uff.br}
\author[L. Portugal]{ Lucas Portugal}
\email{lucasportugal@id.uff.br}
\author[R. Del-Vecchio]{Renata Del-Vecchio}
\email{rrdelvecchio@id.uff.br}
\author[R. Freitas]{Renata de Freitas}
\email{renatafreitas@id.uff.br}

\address{Instituto de Matemática e Estatística, Universidade Federal Fluminense, Niter\' oi, RJ, Brazil}
\begin{document}

\begin{abstract}
In this article we introduce a definition of \emph{$k$-uniform thresholds hypergraphs} through a binary sequence, a natural extension of the classical definition for thresholds graphs. 
We characterize some of its eigenvalues and multiplicities by means of combinatorial numbers, derived from edge counts.
An important problem addressed in Spectral Graph Theory
is to find graphs with few distinct eigenvalues.
Our characterization allows us to construct $k$-uniform threshold
hypergraphs having an arbitrary number of vertices with few distinct
eigenvalues.

\end{abstract}
\maketitle
\section{Introduction}

\label{sec:int}

Although the study of hypergraphs is well-established from a structural perspective, a spectral theory for hypergraphs is relatively new and follows two possible approaches: through matrices or tensors. The tensor method was proposed by Cooper and Dutle \cite{Cooper} in 2012, and initially seemed more promising. However, due to the challenges in computing their eigenvalues, the interest in the matrix-based approach has grown recently, as evidenced by \cite{Banerjee-matriz,energy,Reff2019,kr-regular,matrix-pec-radius-2017,estrada-index,distance}.
The main criticism of using matrices is that a single matrix can represent different hypergraphs, which does not happen with tensors. However, when restricted to certain classes of hypergraphs, each adjacency matrix corresponds to a unique hypergraph, allowing the matrix spectrum to reflect structural properties, as in spectral graph theory.

In this work, we introduce a new class of hypergraphs, the binary  $k$-uniform threshold hypergraphs, which generalize threshold graphs. Among some important structural properties of these hypergraphs, we prove that, in this class, each matrix corresponds to a unique hypergraph (Proposition~\ref{matrix_sequence}). Based on this fact, we investigate the spectrum of these hypergraphs, characterizing some of their eigenvalues and respective multiplicities.

Obtaining the spectrum of binary $k$-uniform threshold hypergraphs allows us to address an important problem from spectral graph theory in this context, namely, finding infinite families of graphs with few distinct eigenvalues.

This paper is structured into four sections. Besides this introduction, in Section 2, we present basic concepts and definitions of hypergraphs. We define the class of $k$-uniform binary threshold hypergraphs and present some structural properties of these hypergraphs. Section 3 contains the main results, where the spectrum is obtained. As an application, we present infinite families of hypergraphs in this class with few distinct eigenvalues. Finally, in the last section, we present the conclusions.

\section{Hypergraphs and Thresholds hypergraphs}
\subsection{Hypergraphs}

In \cite{Bretto}, the basic definitions and concepts of hypergraph theory can be found. The necessary  definitions for understanding the remainder of the article will be presented below, to ensure the text is self-contained.

\begin{definition}
A \emph{hypergraph} $\h=(V,E)$ is given by a finite vertex set $V$ and a set $E=\{e:e\subseteq V\}$, whose elements are called (hyper) edges and have cardinality greater than or equal $2$.  

A hypergraph $\h$ is said to be a \emph{$k$-uniform} hypergraph (or a $k$-graph), for $k\geq 2$, if all edges have the same cardinality $k$.   
\end{definition}

A graph is a special case of a $k$-uniform hypergraph, considering $k=2$.
\begin{definition}
Let $\h$ be a hypergraph with $n$ vertices. The \emph{adjacency matrix} of $\h$, denoted by $\A(\h)$, is the $n \times n$ symmetric matrix with entries $$a_{i,j}=|\left\{ e \in E(\h):v_i,v_j\in e \right\}|.$$ We also denote the eigenvalues of $\A(\h)$ as $\lambda_i(\h), 1\leq i\leq n.$ 

The \emph{spectrum} of $\A(\h)$ is defined by $$spec (\h)= \{ (\lambda_1(\h))^{m_1}, (\lambda_2(\h))^{m_2}, \ldots,  (\lambda_t(\h))^{m_t}\},$$ 
where $\lambda_1(\h) \geq \cdots \geq \lambda_t(\h)$ are the distinct eigenvalues of $\A(\h)$, and $m_1, \ldots, m_t$ denote their respective multiplicities.

\end{definition}
\medskip
\subsection{Thresholds hypergraphs}
A threshold graph may be defined in many ways, for example, as a graph free from $\{2K_2, C_4, P_4\}$, or as a split cograph. For us, the definition through a binary sequence will be more interesting: using this definition, we are able to generalize some results from threshold graphs to $k$-uniform threshold hypergraphs.

\
The definition of threshold graphs through a binary sequence
is extended for the case of $k$-uniform hypergraphs as follows.
We introduce an isolated vertex for each $0$ in the sequence, 
and a vertex connected to all previous vertices in all possible ways for each $1$ in the sequence.

\begin{definition}\label{binary}
Let $k$ be a positive integer and $(b_1,b_2,\ldots,b_n)_k$ be a binary sequence (for each $i$, with $1\leq i\leq n$, we have $b_i\in\{0,1\}$) such that $b_1=b_2=\cdots=b_{k-1}=0$. The $k$-uniform threshold hypergraph $\h=(V,E)$ with (ordered) vertex set $V=\{v_1,v_2,\ldots,v_n\}$ given by $(b_1,b_2,\ldots,b_n)_k$ is defined as follows:
$$e\in E\Longleftrightarrow e\mbox{ is a $k$-subset of }V\mbox{ and if }j=\mathsf{max}\{i:v_i\in e\}\mbox{ then }b_j=1.$$
\end{definition}

Given a $k$-uniform threshold hypergraph $\h=(V,E)$ and its binary sequence $(b_1,\ldots,b_n)_k$,
we denote by $v_i\in V$ the vertex added at the step $i$ in the recursive construction of $\h$, and for $v_i,v_j\in V$ we define $v_i<v_j$ when $i<j$, that is, when vertex $v_i$ was added before vertex $v_j$. In general, for $v, w \in V$, we write $v < w$ whenever vertex $v$ precedes vertex $w$ in the recursive construction of $\h$ from the associated binary sequence.

If $e=\{x_1,x_2,\ldots,x_k\}$ is an edge of $\h$ we will assume that $x_1<x_2<\cdots<x_k$ and we say that the edge $e=\{x_1,x_2,\ldots,x_k\}$ \emph{ends} at $x_k$.
If $b_i=1$, we say that $v_i$ is a \emph{pseudodominant} vertex.
Pseudodominant vertices are connected to every previous $v_1,v_2,\ldots,v_{i-1}$ vertices in every possible way, that is,
if $b_i=1$, then $e\cup\{v_i\}$ is an edge, for every possible $(k-1)$-subset $e$ of $\{v_1,v_2,\ldots,v_{i-1}\}$. 
If $b_i=0$ and $\{x_1,\ldots,x_k\}$ is an edge satisfying $x_1<x_2<\cdots<x_k$, then $x_k\neq v_i$.
Note that if some edge ends at a vertice $v_i$, then $b_i=1$ 
(otherwise, $b_i=0$ and $v_i$ is added as an isolated vertex at step $i$, and so it is not connected to any previous added vertex).

In the remaining of the text we will call a \emph{binary threshold hypergraph} only by threshold hypergraph.

\begin{observacao} 
With this definition, a $2$-uniform threshold hypergraph is a threshold graph.
\end{observacao}

\begin{exemplo}
Let $\h=(V,E)$ be a $4$-uniform threshold hypergraph on $6$ vertices, say $V=\{1,2,3,4,5,6\}$, 
with binary sequence $(0,0,0,1,1,0)_4$. Then
$$E=\{\{1,2,3,4\},\{1,2,3,5\},\{1,2,4,5\},\{1,3,4,5\},\{2,3,4,5\}\}.$$
\end{exemplo}

In \cite{threshold}, Reiterman, R\"odl, \u{S}i\u{n}ajov\'a and T\r{u}ma propose generalizations for $k$-uniform threshold hypergraphs $\h=(V,E)$ from three equivalent characterizations of threshold graphs, as follows:

\begin{enumerate}
\item There exists a positive integer labeling $c$ of $V$ and a positive integer threshold $t$ such that, for all subsets $X\subseteq V$, we have that $X$ contains a member of $E$ if and only if $$\displaystyle\sum_{x \in X} c(x)>t.$$

\item There exists a positive integer labeling $c'$ of $V$ and a positive integer threshold $t'$ such that, for all $k$-subsets $A\subseteq V$, we have that $A\in E$ if and only if $$\displaystyle\sum_{x \in A} c'(x)>t'.$$
    
\item\label{deft} For $x,y \in V$, define $x\ll y$ if $x$ can be replaced by $y$ in any edge. That is,
for any $\{x_1,x_2,\ldots,x_{k-1}\} \subseteq V-\{x,y\}$, we have that $$\{x,x_1,x_2,\ldots,x_{k-1}\} \in E\mbox{ implies }\{y,x_1,x_2, \ldots,x_{k-1}\} \in E.$$ 
Then, for all $x,y \in V$ either $y\ll x$ or $x\ll y$. 
\end{enumerate}

\begin{observacao}\label{obs_threshold}
The implications $(1)\Rightarrow (2)\Rightarrow (3)$ holds for any $k\geq 2$. The reverse implications holds only for $k=2$ ~\cite{threshold}, that is, the three formulations are equivalent only for graphs.
\end{observacao}

In what follows we show that Definition~\ref{binary} implies characterization (\ref{deft}), and also that the reverse implication is false.     

\begin{proposicao}\label{prop:good_def}
Let $\h=(V,E)$ be a $k$-uniform threshold hypergraph defined by a binary sequence. Then for all $x,y\in V$, either $y\ll x$ or $x\ll y$.
\end{proposicao}
\begin{proof}
W.l.o.g., let $x,y$ be two vertices such that $x<y$.

Suppose $y\not\ll x$.
Then, there is some $\{y_1,\ldots,y_{k-1}\}\subseteq V-\{x,y\}$ such that $f=\{y,y_1,\ldots,y_{k-1}\}\in E$
but $f'=\{x,y_1,\ldots,y_{k-1}\}\not\in E$.
Note that edge $f$ ends at $y$.
(If not, edge $f$ would end at some $y'\in\{y_1,\ldots,y_{k-1}\}$ and $y'$ would be pseudodominant. 
But then $f'$ would be a $k$-subset of $V$ that ends at the same pseudodominant vertex $y'$ (since $x<y<y'$),
contradicting the fact that $f'$ is not an edge.)
Hence, $y$ is pseudodominant.

Let $\{x_1,\ldots,x_{k-1}\}\subseteq V-\{x,y\}$, with $x_1<\cdots<x_{k-1}$. 
Suppose $e=\{x,x_1,\ldots,x_{k-1}\}\in E$.
We have two cases: (1) edge $e$ ends at $x$, or (2) edge $e$ ends at $x_{k-1}$.
\begin{enumerate}
\item 
Edge $e$ ends at $x$. 
In this case, since $e\in E$, then $x$ is pseudodominant.
Then $e'=\{y,x_1,\ldots,x_{k-1}\}$ ends at $y$ (since $x<y$).
Hence, since $y$ is pseudodominant, $e'\in E$.
\item
Edge $e$ ends at $x_{k-1}$.
In this case, since $e\in E$, then $x_{k-1}$ is pseudodominant.
Now, if $y<x_{k-1}$, then $e'=\{y,x_1,\ldots,x_{k-1}\}$ ends at a pseudodominant vertex and $e'\in E$;
if $y>x_{k-1}$, then $e'=\{y,x_1,\ldots,x_{k-1}\}$ also ends at a pseudodominant vertex (since $y$ is pseudodominant) and $e'\in E$.
\end{enumerate}
In any case, $e'=\{y,x_1,\ldots,x_{k-1}\}\in E$, which proves that $x\ll y$.
\end{proof}

The above proposition tells us that our definition is good in the sense that the set of hypergraphs given by the binary sequence (Definition~\ref{binary}) is a subset of the hypergraphs satisfying the characterization (\ref{deft}) from 1985, which is also one of the oldest and most used characterization of uniform threshold hypergraphs \cite{2020threshold}.

In other words, from Proposition~\ref{prop:good_def}, every $k$-uniform threshold hypergraph given by a binary sequence satisfies characterization (\ref{deft}). But the reverse implication is not true, as shown in the following example. 

\begin{exemplo}
In fact, let $\h=(V,E)$ be the $4$-uniform hypergraph on $7$ vertices with set of vertices $V=\{1,2,\ldots,7\}$ and set of edges
$$E=\{\{1,5,6,7\},\{2,5,6,7\},\{3,5,6,7\},\{4,5,6,7\}\}.$$
Then $\h$ satisfies characterization (\ref{deft}): 
\begin{enumerate}
\item We have that $x\ll y$, for any $x,y \in \{1,2,3,4\}$ by vertex replacement in the edges. 
\item  The vertices $5,6$ and $7$ belongs to every edge of $\h$, then there are no subset with $3$ elements of $V-\{x,y\}$ with $x,y \in \{5,6,7\}$ that are contained in $E$. Therefore we have that $x\ll y$, for any $x,y \in \{5,6,7\}$. 
\item Now it suffices to show that $5\ll 4$ (because cases $x\ll y$ for $x \in \{5,6,7\}$ and $y \in \{1,2,3,4\}$ are analogous). 

The subsets of $V-\{4,5\}$ with $3$ elements that are contained in $E$ are $\{1,6,7\}$, $\{2,6,7\}$, and $\{3,6,7\}$.

By definition, $5\ll 4$ is true if for any of $\{1,4,6,7\}$, $\{2,4,6,7\}$, or $\{3,4,6,7\}$ that is an edge then we must have that $\{1,5,6,7\},\{2,5,6,7\}$ or $\{3,5,6,7\}$ is also an edge.
But since the subsets $\{1,4,6,7\},\{2,4,6,7\}$ or $\{3,4,6,7\}$ are not edges, the condition $5\ll 4$ holds.
\end{enumerate}

Since we have $x\ll y$ or $y\ll x$ 
for any $x,y\in V$, then $\h$ satisfies characterization (\ref{deft}).
However, $\h$ is not represented by any binary sequence $(b_1,\ldots,b_7)_4$, because, if it was, it would have some a dominating vertex $v_7$ corresponding to $b_7=1$ (since it is connected). But then $E$ would have at least $20$ edges --- corresponding to all the $\binom{6}{3}$ edges containing $v_7$.
\end{exemplo}

Next,  we generalize some important structural properties of threshold graphs to the class of $k$-uniform threshold hypergraph.

\begin{proposicao} Let $\h$ be a  $k$-uniform threshold hypergraph, Then
\begin{itemize}
    \item[1.] Any induced subhypergraph of $\h$ is also a $k$-uniform threshold hypergraph.
    \item[2.] The complement of a $\h$ is a $k$-uniform threshold hypergraph.
     \item[3.] A threshold hypergraph is a split-hypergraph.

\end{itemize}

\end{proposicao}
\begin{proof}
\begin{itemize}
    \item[1.] It suffices to show that by removing a single vertex of a $k$-uniform threshold hypergraph we obtain a $k$-uniform threshold hypergraph. That can be verified as follows.

If we remove a vertex $i$ corresponding to a binary entry $1$, 
then will simply remove the corresponding entry from the binary sequence.
If we remove a vertex $i$ corresponding to a binary entry $0$, then we have two cases:
if $i\leq k-1$ and the $k$-th entry in the binary sequence is $1$, removing $i$ will remove the $k$-th and first entry $1$ in the binary sequence;
otherwise, the corresponding entry $0$ in the binary sequence vanishes.

 \item[2.] If $\h$ is a $k$-uniform threshold hypergraph, then its complement has the binary sequence where the first $k-1$ entries remain $0$ and, after that, every entry $1$ in the binary sequence defining $\h$ becomes $0$ in its complement and every entry $0$ 
(after the first $k-1$ entries) becomes $1$.

 \item[3.] The set S of all vertices with entry 0 on the binary sequence is a stable set and the set C of all vertices with entry 1 on the binary sequence is a clique.
 
\end{itemize}

\end{proof}

The last  proposition indicates that the definition proposed here is an appropriate generalization of the concept of threshold graphs. Threshold hypergraphs have structural properties similar to those of threshold graphs, preserving the threshold condition in induced subhypergraphs and complement, as well as maintaining the property of being split.

For graphs in general, the adjacency matrix determines the graph univocally.
The same occurs for $k$-uniform threshold hypergraphs:
the adjacency matrix of a $k$-uniform threshold hypergraph determines the binary sequence associated to the hypergraph in a unique way.

\begin{proposicao}\label{matrix_sequence}
Let $\h$ and $\h^\prime$ be $k$-uniform threshold hypergraphs  given by the binary sequences $(b_1,\dots,b_n)_k$ and $(b_1^\prime,\dots,b_n^\prime)_k$, respectively, if $\h$ and $\h'$ have the same adjacency matrices, then $(b_1,\dots,b_n)_k=(b_1^\prime,\dots,b_n^\prime)_k$.
\end{proposicao}

\begin{proof} Let $\h$ and $\h^\prime$ be two $k$-uniform threshold hypergraphs given by the binary sequences $(b_1, \dots, b_n)_k$ and $(b_1^\prime,\dots,b_n^\prime)_k$, respectively.
If the sequences are not the same, then neither will be the adjacency matrices, say $A$ and $A^\prime$, respectively.
In fact, let $j=\mathsf{max}\{i; b_i\neq b_i^\prime\}$ and suppose that $b_j=1$ and $b^\prime_j=0$. We can suppose that $k\leq j<n$ since if $j=n$ then one of the hypergraphs is connected and the other is not. But in this case the last row of the adjacency matrices will be different.
We claim that $a_{1, j}\neq a^\prime_{1, j}$.

We have two types of edges containing $v_1$ and $v_j$ in $\h$:
edges Type 1, those that contain $v_1,v_j$ and $v_l$ for some $l>j$ 
(that is, edges not ending at $v_j$)
and edges Type 2, those that do not contain any vertex added after $v_j$ 
(that is, edges ending at $v_j$). 
In $\h^\prime$ the only edges containing $v_1$ and $v_j$ are of the first type, 
since $b^\prime_j=0$ does not allow to have an edge ending at $v_j$. 
Note that the cardinality of the sets of Type 1 edges in both hypergraphs are the same, since $b_i=b_i^\prime$ for all $i>j$. Let's denote by $T_1$ and $T_2$ the cardinality of the sets of Type 1 and Type 2 edges in $\h$, respectively.
Then
$$a_{1,j}=T_1+T_2=T_1+\binom{j-2}{k-2}{>} T_1=a^\prime_{1,j}.$$
\end{proof}

In other words, from Proposition~\ref{matrix_sequence}, it can be deduced that each matrix correspond to a unique $k$-uniform threshold hypergraph, which makes the study of threshold hypergraphs based on matrices more interesting, as mentioned in the introduction.

\section{Spectral properties of threshold hypergraphs}

We recall that threshold graphs have well-known spectral properties, considering the adjacency matrix, as can be seen in \cite{irene} and \cite{trevisan}. Results on eigenvalues of other matrices associated with threshold graphs are also known, see for example \cite{Hammer,Vinagre}.

We begin this section by introducing a useful notation.

\begin{definition}
Let $\h$ be a $k$-uniform threshold hypergraph given by the binary sequence $(b_1, b_2, \dots, b_n)_k$.
We define $C(a_1, a_2, \dots, a_r)_k$, the \emph{short sequence} associated to hypergraph $\h$, as follows:
\begin{itemize}
\item If $b_k=0$, we define $a_1$ as the number of zeros at the first block at beginning of the binary sequence, $a_2$ as the number of ones at the second block of the binary sequence, then $a_3$ as the number of zeros in the third  block of the binary sequence, and so on:
\[
a_{2t+1}=\#\{i;b_j=0\mbox{ for all $j$ such that  } {\sum_{l=1}^{2t}a_l}<j\leq i\}
\]
and
\[
a_{2t}=\#\{i;b_j=1\mbox{  for all $j$ such that }{\sum_{l=1}^{2t-1}a_l}<j\leq i\},
\]
for $0\leq t\leq\frac{(r-1)}{2}$ (if $r$ is odd) or $0\leq t\leq\frac{r}{2}$ (if $r$ is even).
\item If $b_k=1$, we define $a_1$ as the number of zeros and ones at the first two blocks of the binary sequence, $a_2$ as the number of zeros in the third block of the binary sequence, and so on.
\end{itemize}
\end{definition}
Note that the short sequence is recursively defined.

From now on, we will identify a $k$-uniform threshold hypergraph $\h$ with its short sequence $C(a_1,a_2,\dots,a_r)_k$, 
using the notation in \cite{irene} mentioned above. 

\begin{exemplo} Let $\h_1$ be the $3$-uniform threshold hypergraph given by the binary sequence $(0,0,0,0,0,1,1,0,1,1,1,0,0,0,1)_3$. We have $b_3=0$, then  
$$(\underbrace{0,0,0,0,0}_5, \underbrace{1,1}_2, \underbrace{0}_1, \underbrace{1,1,1}_3, \underbrace{0,0,0}_3, \underbrace{1}_1)$$
and the short sequence associated to $\h_1$ is  $C(5,2,1,3,3,1)_3.$

 Let  $\h_2$ be the $4$-uniform threshold hypergraph  given by the binary sequence $(0,0,0,1,1,1,1,0,1,1,0,0,0,1)_4$. We have $b_4=1$, then
$$(\underbrace{0,0,0,1,1,1,1}_7, \underbrace{0}_1, \underbrace{1,1}_2, \underbrace{0,0,0}_3, \underbrace{1}_1)$$
and and the short sequence associated to $\h_2$ is    $C(7,1,2,3,1)_4.$
\end{exemplo}

In the case of threshold graphs, that is, $2$-uniform threshold hypergraphs, 
this short sequence coincides with the sequence introduced in \cite{irene}.

We may restrict our attention to connected hypergraphs, hence $b_n = 1$. 
Note that $r$ is even when $b_k = 0$ and odd when $b_k = 1$. 

\begin{definition} Let $\h$ be a $k$-uniform threshold hypergraph given by short sequence $C(a_1,a_2,\dots,a_r)_k$. 
We say that the vertices $v_i$ and $v_j$ belong to the same \emph{block} if there exists $t$ with $1\leq t\leq r$ such that 
$$\displaystyle\sum_{l=1}^{t-1}a_l<i,j\leq\sum_{l=1}^{t}a_l.$$
In this case, we say that $v_i$ and $v_j$ belong to the $t$-th block.
\end{definition}

In standard graph terminology, vertices within the same block are referred to as twins. If the vertices are in a same block of ones, they are called true twins; if they are in a block of zeros, they are called false twins.

\begin{proposicao} \label{intercambio}
Let $\h$ be a $k$-uniform threshold hypergraph given by short sequence $C(a_1,a_2,\dots,a_r)_k$. 
If $x$ and $y$ are vertices belonging to the same block, then $x\ll y$ and $y\ll x$.
\end{proposicao} 

\begin{proof} 
Let $\h=(V,E)$ be a $k$-uniform threshold hypergraph. 
Let $x$ and $y$ be vertices in $\h$ belonging to the same block.
Let $\{x_1,\ldots,x_{k-1}\}\subseteq V-\{x,y\}$, with $x_1<\cdots<x_{k-1}$. 
Since $x$ and $y$ belong to the same block,
$x$ is pseudodominant if, and only if, $y$ is pseudodominant.
Assume w.l.o.g. that $x<y$. 
We have three cases: 
(1) $x<y<x_{k-1}$, or 
(2) $x_{k-1}<x<y$, or 
(3) $x<x_{k-1}<y$.
\begin{enumerate}
\item
($x<y<x_{k-1}$)

In this case, both 
$\{y,x_1,\ldots,x_{k-1}\}$
and
$\{x,x_1,\ldots,x_{k-1}\}$
end at vertex $x_{k-1}$. Then
$\{x,x_1,\ldots,x_{k-1}\}\in E$ $\Longleftrightarrow$
$x_{k-1}$ is pseudodominant $\Longleftrightarrow$
$\{y,x_1,\ldots,x_{k-1}\}\in E$.
\item
($x_{k-1}<x<y$)

In this case, 
$\{y,x_1,\ldots,x_{k-1}\}$
ends at $y$ and
$\{x,x_1,\ldots,x_{k-1}\}$
ends at $x$. Then
$\{x,x_1,\ldots,x_{k-1}\}\in E$ $\Longleftrightarrow$
$x$ is pseudodominant $\Longleftrightarrow$
$y$ is pseudodominant $\Longleftrightarrow$
$\{y,x_1,\ldots,x_{k-1}\}\in E$.
\item
($x<x_{k-1}<y$)

In this case, 
$\{y,x_1,\ldots,x_{k-1}\}$
ends at $y$ and
$\{x,x_1,\ldots,x_{k-1}\}$
ends at $x_{k-1}$. 
Besides that, since $x$ and $y$ belong to the same block,
$x_{k-1}$ and $y$ also belong to the same block.
Hence, 
$x_{k-1}$ is pseudodominant if, and only if, $y$ is pseudodominant.
Then
$\{x,x_1,\ldots,x_{k-1}\}\in E$ $\Longleftrightarrow$
$x_{k-1}$ is pseudodominant $\Longleftrightarrow$
$y$ is pseudodominant $\Longleftrightarrow$
$\{y,x_1,\ldots,x_{k-1}\}\in E$.
\end{enumerate}
In any case,
$\{x,x_1,x_2,\dots,x_{k-1}\}\in E\Longleftrightarrow\{y,x_1,x_2,\dots,x_{k-1}\}\in E$.
\end{proof}

As a corollary we have that columns $i$ and $j$ of the adjacency matrix $A(\h)$ of $\h$ differ only in the elements of the diagonal. 

\begin{corolario} \label{colunas} 
Let $\h$ be a $k$-uniform threshold hypergraph given by a binary sequence 
and $A(\h)=(a_{i,j})$ its adjacency matrix.
If $v_i$ and $v_j$ are vertices belonging to the same block, then $a_{s,i}=a_{s,j}$ for all $s\neq i,j$. 
\end{corolario}

\begin{proof} 
Let $s\neq i,j$ and define $E(a,b):=\{e\in E;v_a,v_b\in e\}$.
Since $a_{s,i}=|E(s,i)|$ and $a_{s,j}=|E(s,j)|$, it is enough to construct an isomorphism between these two sets.

We define $\varphi:E(s,i)\to E(s,j)$ to be such that $\varphi(e)=(e-\{v_i\})\cup\{v_j\}$,
that is, we are taking each edge $e$ that contains $v_i$ and $v_s$ and replacing $v_i$ with $v_j$. 
The new set $(e-\{v_i\})\cup\{v_j\}$ contains $v_s$ and $v_j$ and it is an edge (by Proposition \ref{intercambio}). 
It is easy to see that $\varphi$ is an isomorphim.
\end{proof}

The next result provides eigenvalues and their multiplicities directly from the entries of the hypergraph adjacency matrix.

\begin{teorema}\label{spec} 
Let $\h$ be a $k$-uniform threshold hypergraph given by short sequence $C(a_1,a_2,\dots,a_r)_k$ 
and $A(\h)=(a_{i,j})$ its adjacency matrix.
Then $-a_{s, s+1}$ is an eigenvalue of $A(\h)$ with multiplicity at least $a_j-1$, where $s=a_1+\cdots+a_{j-1}+1$. 
That is, $s$ is the position of the first vertex in the $j$-th block. 
\end{teorema}

\begin{proof} 
Let's denote by $e_{a,b}$ the $n$-dimensional vector such that the coordinate $a$ is equal to $1$, the coordinate $b$ is equal to $-1$, and the others are zero.
The vectors $e_{s,s+1},\dots,e_{s,s+a_j-1}$ are eigenvectors associated with $-a_{s, s+1}$, this follows immediately from the Corollary \ref{colunas}. Moreover, they are linearly independent vectors, which concludes the proof.
\end{proof}

\begin{corolario} Let $\h$ be a $k$-uniform threshold hypergraph given by short sequence $C(a_1,a_2,\dots,a_r)_k$. 
Then Theorem~\ref{spec} provides $n-r$ eigenvalues of $A(\h)$.
\end{corolario}

In fact, each block $i$ provides $a_i-1$ eigenvalues, as $a_1+\dots+a_r=n$, then it follows  trivially  that $(a_1-1)+\dots+(a_r-1)=n-r.$

Now we characterize the eigenvalues of $A(\h)$ given in Theorem \ref{spec}.
Let us begin introducing some notations.
 
Let $\h$ be a $k$-uniform threshold hypergraph given by $C(a_1,a_2,\dots,a_r)_k$. 
If $a_j\geq 2$ for some $j$, then $-a_{s, s+1}$ is an eigenvalue of $A(\h)$, where $s=a_1+\cdots+a_{j-1}+1$, from Theorem \ref{spec}.   Remembering that $a_{s, s+1}=|E(s,s+1)|$ we need to compute the number of edges containing $v_s$ and $v_{s+1}$.
 
We first analyze the case in which $b_k=0$. In this case $r$ is an even number, say $r=2m$.
\begin{itemize}
\item 
If $a_j\geq 2$ and $j$ is an odd number, say $j=2t-1$ for some $1\leq t\leq m$.
In this case, vertices $v_s$ and $v_{s+1}$ belong to the same block of zeros. 
Then if $e\in E$ is an edge containing both vertices we have that $e$ must end at a pseudodominant vertex belonging to a posterior block, that is, one of the blocks ${2t},{2t+2},\dots,{2m}$.

The numbers of edges $e$ that ends at vertices belonging to block $2\ell$ are:

\[
N_{2\ell}=\sum_{j=1}^{a_{2\ell}}\binom{a_1+\cdots+a_{2\ell}-3+j}{k-3},
\]

 where $\binom{a_1 + \cdots + a_{2\ell} - 3 + j}{k - 3}$ is equal to the  cardinality of the set of edges that contain the vertices $v_s$ and $v_{s+1}$ and end at the $j^{\text{th}}$ vertex of the $2\ell$-block.

Then with the previous notation we have that:
$$
-\bigg(\sum_{\ell=t}^{m}N_{2\ell}\bigg)
$$
is an eigenvalue of $A(\h)$ with multiplicity at least $a_{2t-1}-1$.

\item 
If $a_j\geq 2$ and $j$ is an even number, say $j=2t$ for some $1\leq t\leq m$, then vertices $v_s$ and $v_{s+1}$ belong to the same block of ones. Then if $e\in E$ is an edge containing both vertices we have that $e$ could end at a vertex that corresponds to a posterior block of ones, that is, associated to one of the blocks ${2t+2},\dots,{2m}$ (Type 1 edges) or $e$ must ends in a vertex belonging to the same block as $v_s$ and $v_{s+1}$ (the $j^{\text{th}}$ block).

The number $T_1$ of Type 1 edges is: 
$$T_1=\sum_{\ell=t+1}^{m}N_{2\ell},$$
where $N_{2\ell}$ is defined as before.

The number of edges ending in a vertex of the same block still needs to be calculated.
In this case:
$$T_2=\binom{a_1+\cdots+a_{2t}-2}{k-2},$$
since the number of such edges is the same of the number of edges ending at  the last vertex on the block $v_{a_1+\cdots+a_{2t}}$ that contains $v_s$ by Proposition~\ref{intercambio}.
 
We have that: 
$$-\displaystyle\left(\sum_{\ell=t+1}^{m}N_{2\ell}\right)-\binom{a_1+\cdots+a_{2t}-2}{k-2}$$
is an eigenvalue of $A$ with multiplicity at least $a_{2t}-1$.
\end{itemize}

When $b_k=1$ and $r=2m-1$ is an odd number, the computations are as in the previous case, taking into account now that the blocks in even positions correspond to zeros in the binary sequences and the ones in odd positions to ones. In this case, for $1\leq \ell\leq m$, we have $N_{2\ell-1}=$ 
$$\binom{a_1+\cdots+a_{2\ell-2}-2}{k-3}+\binom{a_1+\cdots+a_{2\ell-2}-1}{k-3}+\cdots+\binom{a_1+\cdots+a_{2\ell-1}-3}{k-3},$$

and

$$T_2=\binom{a_1+\cdots+a_{2t-1}-2}{k-2}.$$

Summarizing, we just prove the following result.

\begin{teorema}\label{caracterizacao} 
Let $\h$ be a connected $k$-uniform threshold hypergraph given by short sequence $C(a_1,a_2,\dots,a_{r})_k.$ For each $j$ such that $a_j\geq 2$ we obtain an eigenvalue of $A(G)$ and its multiplicity as follows:
\begin{itemize}
\item[\rm 1)] Suppose that $r=2m$ (here $b_k=0$). Let $N_{2\ell}$ be as before.
\end{itemize}

\begin{itemize}
\item[\rm a)] If $j=2t-1$ for some integer $t$, then $-\bigg(\displaystyle\sum_{\ell=t}^{m}N_{2\ell}\bigg)$ is an eigenvalue of $A(\h)$
with multiplicity at least $a_{2t-1}-1$.
\item[\rm b)] If $j=2t$ for some integer $t$, then $-\Bigg(\displaystyle\sum_{\ell=t+1}^{m}N_{2\ell}+\binom{a_1+\cdots+a_{2t}-2}{k-2}\Bigg)$ is an eigenvalue of $A(\h)$ with multiplicity at least $a_{2t}-1$.
\end{itemize}

\begin{itemize}
\item[\rm 2)] Suppose that $r=2m-1$ (here $b_k=1$). Let $N_{2\ell-1}$ be as before.
\end{itemize}
\begin{itemize}
\item[\rm a)] If $j=2t-1$ for some integer $t$, then 
$$-\Bigg(\displaystyle\left(\sum_{\ell=t+1}^{m}N_{2\ell-1}\right)+\binom{a_1+\cdots+a_{2t-1}-2}{k-2}\Bigg)$$ 
is an eigenvalue of $A(\h)$ with multiplicity at least $a_{2t-1}-1$.
\item[\rm b)] If $j=2t$ for some integer $t$, then 
$-\bigg(\displaystyle\sum_{\ell=t+1}^{m}N_{2\ell-1}\bigg)$
is an eigenvalue of $A(\h)$ with multiplicity at least $a_{2t}-1$.
\end{itemize}
\end{teorema}

The others $r$ missing eigenvalues will be the eigenvalues of $B$, the equitable quotient matrix obtained from $A(\h)$ by taking the partition
$$
A(\h)=\begin{pmatrix} A_{11}&\cdots&A_{1r}\\
\vdots&\ddots&\vdots\\
A_{1r}&\cdots&A_{rr}
\end{pmatrix}
$$
where the blocks $A_{ij}$ are $a_i\times a_j$ matrices. Note that, for each submatrix $A_{ij}$, the sum of the elements in each row is constant (as a consequence of Corollary \ref{colunas}). This means that all vertices in the same block are twins, therefore they have the same neighbors.

\vspace{0.9mm}

We will finish this section by calculating the spectrum of three families of $k$-uniform threshold hypergraphs. These families have a small number of eigenvalues despite having an arbitrary number of vertices.

\subsection{Infinite family of $k$-uniform threshold hypergraphs with at most $3$ distinct eigenvalues}

Consider the class of $k$-uniform threshold hypergraphs given by the binary sequence $(0,\ldots,0,1)_k$, i.e., a class consisting of binary sequences where all entries are $0$ except the $n^{th}$ and last entry which is $1$. When $n = k$, the hypergraph $(0,\ldots,0,1)_k$ has only one edge containing all vertices.

If $k<n$, then $\h=C(n-1, 1)_k$, using the previous notation we have $r=2, m=1, t=1$ and $a_1\geq 2$. By Theorem \ref{caracterizacao} we have that:
$$-\bigg(\sum_{\ell=t}^{m}N_{2\ell}\bigg)=-\binom{n-3}{k-3}$$
is an eigenvalue of $A(\h)$ with multiplicity at least $n-2$.

The adjacency matrix $A(\h)$ can be partitioned as  
$$\A(\h) = \left[\begin{array}{ccc}
\binom{n-3}{k-3}(\J-\I)_{(n-1 \times n-1)} & | & \binom{n-2}{k-2}\J_{(n-1 \times 1)} \\ 
---------- & | & ------  \\ 
\binom{n-2}{k-2}\J_{(1 \times n-1)} & | & \mathbf{0}_{(1 \times 1)} 
\end{array}\right]. $$

In this case the quotient matrix $B$ is 
$$B=\left[\begin{array}{cc}
a(n-2) & b  \\ 
b(n-1) & 0 
\end{array}\right],$$ 
whose characteristic polynomial is $f(x)= x^2-a(n-2)x-b^2(n-1)$, with $a=\binom{n-3}{k-3}$ and $b=\binom{n-2}{k-2}$. Therefore its roots
$\alpha$ and $\beta$ are eigenvalues of $\A(\h)$.

We have that $\mathsf{spec}(\h)= \left\{\alpha,-\binom{n-3}{k-3}^{(n-2)},\beta\right\}$.

\begin{exemplo}
Consider the $3$-uniform threshold hypergraph $\h=(V,E)$, with $V=\{1,2,3,4,5\}$ and $$E= \{\{1,2,5\}\},\{1,3,5\},\{1,4,5\},\{2,3,5\},\{2,4,5\},\{3,4,5\}\}.$$ Its binary sequence is $(0,0,0,0,1)_3$. We have that: 
$$\mathsf{spec}(\h)=\{(7,68)^1, (-1)^3, (-4,68)^1\}.$$ 
\end{exemplo} 

\subsection{Infinite family of $k$-uniform threshold hypergraphs with at most 4 distinct eigenvalues}

Fix $j$, where $k \leq j \leq n-1$. This subclass consists of binary sequences where the first $j-1$ entries are $0$. After that, every entry is $1$: $(0,\dots, 0, 1,\dots, 1)_k$. Denote by $q$ the quantity of $1$ in the sequence, that is $q=n+1-j$. 
When $j=k$, ie, the $k^{th}$ entry is $1$, then the binary sequence $(0,\ldots,0,1,\ldots,1,1)_k$ is the complete $k-$uniform hypergraph on $n$ vertices.
\vspace{0.3cm}

If $j>k$ then $\h=C(j-1, n-j+1)_k$, that is $a_1=j-1\geq 2$ and $a_2=n-j+1\geq 2$. 
Using the previous notation we have that $r=2$ and $m=1$. Then 
\begin{itemize}
\item[\rm a)] $-\bigg(\sum_{\ell=1}^{1}N_{2\ell}\bigg)=\binom{j-3}{k-3} + \binom{j-2}{k-3} +\cdots+ \binom{n-3}{k-3} =\sum_{i=1}^q \binom{n-(i+2)}{k-3}$
is an eigenvalue of $A(\h)$ with multiplicity at least $j-2$.
\item[\rm b)] $-\sum_{\ell=2}^{1}N_{2\ell}-\binom{a_1+\cdots+a_{2t}-2}{k-2} = -\binom{n-2}{k-2} \text{ is an eigenvalue of $A(\h)$ }$
with multiplicity at least $n-j$.
\end{itemize}

The adjacency matrix is $\A(\h)=$
$$\left[\begin{array}{ccc}
 \left( \sum_{i=1}^q \binom{n-(i+2)}{k-3} \right)(\J-\I)_{(n-q \times n-q)} & | & \binom{n-2}{k-2}\J_{(n-q \times q)} \\ 
------------- & | & ------  \\ 
\binom{n-2}{k-2}\J_{(q \times n-q)} & | & \binom{n-2}{k-2} (\J-\I)_{(q \times q)} 
\end{array}\right]. $$

The quotient matrix $B$ of $\A(\h)$ is the $2\times 2$ matrix
$$B=\left[\begin{array}{cc}
a(n-q-1) & bq  \\ 
b(n-q) & b(q-1) 
\end{array}\right], $$ 
with $a=\left( \sum_{i=1}^q \binom{n-(i+2)}{k-3} \right)$ and $b=\binom{n-2}{k-2}$.

The characteristic polynomial of $B$ is 
$$f(x)=x^2-x\left(a(n-q-1)+b(q-1)\right)+a(n-q-1)b(q-1)-b^2q(n-q)$$
and we will denote by $\alpha$ and $\beta$ its roots.

We prove that 
$$\mathsf{spec}(\h)=\left\{\alpha,-\left(\sum_{i=1}^q \binom{n-(i+2)}{k-3}\right)^{(n-q-1)},-\binom{n-2}{k-2}^{(q-1)},\beta\right\}.$$

\begin{exemplo}
Consider the $3$-uniform threshold hypergraph $\h=(V,E)$ with $V=\{1,2,3,4,5\}$ and 
$$\begin{array}{r}
E=\{\{1,2,4\},\{1,3,4\},\{2,3,4\},\{1,2,5\},\{1,3,5\},\\
\{1,4,5\},\{2,3,5\},\{2,4,5\},\{3,4,5\}\}.
\end{array}$$
Its binary sequence is $(0,0,0,1,1)_3$. Its spectrum is given by
$$\mathsf{spec}(\h)=\{(10,86)^1,(-2)^2,(-3)^1,(-3,86)^1\}.$$ 
\end{exemplo}


\subsection{Infinite family of k-uniform threshold hypergraphs with at most 5 distinct eigenvalues}

This subclass consists of $k-$uniform thresholds given by binary sequences where the first $k-1$ entries are $0$, then the $k^{th}$ entry is $1$, then only zeros until the $n^{th}$ and last entry, which is $1$: $(0,\dots, 0, 1, 0,\dots, 0, 1)_k$.

We must have $n\geq k+2$, because we expect to exist at least one entry $0$ between the two entries $1$, since the hypergraphs $(0,\ldots,0,1)_k$ for $n=k$ and $(0,\ldots,0,1,1)_k$ for $n=k+1$ belongs to previously studied classes. 

In this case we have that $\h$ is given by $C(k,n-k-1, 1)_k$, where $n-k-1\geq 1$. 
From Theorem \ref{caracterizacao}, $-a_{12}$ is an eigenvalue of $A$ with multiplicipy at least $k-1$ and $-a_{(k+1)(k+2)}$ is an eigenvalue of $A(\h)$ with multiplicity at least $n-k-2$.

Let's calculate these two entries of $A(\h)$:
$$
a_{1,2}=|\{e\in E; v_1, v_2\in e\}|=\underbrace{\binom{k-3}{k-3}}_\text{edges ending at $v_k$}+\underbrace{\binom{n-3}{k-3}}_\text{edges ending at $v_n$}=1+\binom{n-3}{k-3}.
$$

Similarly we have 
$$
a_{k+1,k+2}=|\{e\in E; v_{k+1}, v_{k+2}\in e\}|=\underbrace{\binom{n-3}{k-3}}_\text{edges ending at $v_n$}=\binom{n-3}{k-3}.
$$

The adjacency matrix $A(\h)$ is 

$$\left[\begin{array}{ccccc}
\left(\binom{n-3}{k-3}+1\right)(\J-\I)_{(k \times k)} & | & \binom{n-3}{k-3}\J_{(k \times n-k-1)} & | & \binom{n-2}{k-2}\J_{(k \times 1)} \\ 
-------- & | & -------- & | & ------- \\
\binom{n-3}{k-3}\J_{(n-k-1 \times k)} & | & \binom{n-3}{k-3}(\J-\I)_{(n-k-1 \times n-k-1)} & | & \binom{n-2}{k-2}\J_{(n-k-1 \times 1)} \\ 
-------- & | & -------- & | & ------- \\ 
\binom{n-2}{k-2}\J_{(1 \times k)} & | & \binom{n-2}{k-2}\J_{(1 \times n-k-1)} & | & \mathbf{0}_{(1 \times 1)} 
\end{array}\right]. $$

The quotient matrix $B$ is a $3\times 3$ matrix given by
$$\left[\begin{array}{ccc}
(a+1)(k-1) & a(n-k-1) & b \\ 
ak & a(n-k-2) & b  \\ 
bk & b(n-k-1) & 0 
\end{array}\right], $$ 
where $a=\binom{n-3}{k-3}$ and $b=\binom{n-2}{k-2}$. Let's denote by  $\alpha$, $\beta$ and $\gamma$ the eigenvalues of $B$.
We have just proven that 
$$\mathsf{spec}(\h)=\left\{\alpha,\beta,-\binom{n-3}{k-3}^{(n-k-2)},-\left(\binom{n-3}{k-3}+1 \right)^{(k-1)},\gamma\right\}.$$

\begin{exemplo}
Consider the $3$-uniform threshold hypergraph $\h=(V,E)$ with $V=\{1,2,3,4,5\}$ and 
$$E=\{\{1,2,3\},\{1,2,5\},\{1,3,5\},\{1,4,5\},\{2,3,5\},\{2,4,5\},\{3,4,5\}\}.$$
Its binary sequence is $(0,0,1,0,1)_3$.
Its spectrum is given by
$$\mathsf{spec}(\h)=\{(8,71)^1, (-0,49)^1, (-2)^2, (-4,22)^1\}.$$ 
\end{exemplo}

\section{Conclusions}

Assuming $\binom{\hfill n}{-1}=0$ (there is no subset of $\{1,\ldots,n\}$ with $-1$ element), notice that e\-very matrix and spectrum above are well represented for graphs ($k$-uniform hypergraphs with $k=2$) as well.

As a consequence of Theorem~\ref{spec}, for $k$-uniform threshold hypergraphs given by $C(a_1,\dots,a_r)_k$, the fewer the blocks and the larger their size, the smaller the number of distinct eigenvalues.

We present three families of $k$-uniform threshold hypergraphs with arbitrary number $n$ of vertices and, for each $n$, considering $1\leq k\leq n-1$ all these hypergraphs have up to $5$ distinct eigenvalues. 

Note that for $r$ to be maximum for a $k$-uniform threshold $G$, the short sequence is of the form $(k,1,1,...,1)$. Again, by Theorem~\ref{spec}, $G$ has an eigenvalue with multiplicity at least $k-1$, which implies that the maximum number of distinct eigenvalues is $n-k+2$.

In all the numerical examples studied, the eigenvalues of the quotient matrix $B$ are simple ones. We know that this always occurs for graphs \cite{trevisan}. Does this hold in general for $k$-uniform threshold hypergraphs?
We leave this question open.

\section*{Acknowledgements:}
Lucas Portugal acknowledges the financial support  by  Coordena\c{c}\~ao de Aperfei\c{c}oamento de Pessoal de N\'ivel Superior - Brasil (CAPES) - Finance Code 001 and Renata R. Del-Vecchio acknowledges the financial support  by  Conselho Nacional de Desenvolvimento Cient\'ifico e Tecnol\'ogico-Brasil (CNPq), grant 308159/2022-5 and 404788/2023-8.

\end{document}